\documentclass[11pt,reqno]{amsart}
\usepackage[tmargin=1in,bmargin=1in,rmargin=1in,lmargin=1in]{geometry}

\usepackage[breaklinks=true]{hyperref}

\usepackage{amsmath,amsthm,amsfonts,amssymb}
\usepackage{enumerate}
\usepackage{xcolor}
\usepackage{mathdots}
\theoremstyle{plain}
\newtheorem{theorem}{Theorem}[section]

\newtheorem{lemma}[theorem]{Lemma}
\newtheorem{prop}[theorem]{Proposition}
\newtheorem{cor}[theorem]{Corollary}
\newtheorem{utheorem}{\textrm{\textbf{Theorem}}}

\newcommand{\IR}{\mathbb{R}}

\theoremstyle{definition}
\newtheorem{defn}[theorem]{Definition}
\newtheorem{notation}[theorem]{Notation}

\newtheorem{rem}[theorem]{Remark}
\newtheorem{example}[theorem]{Example}

\numberwithin{equation}{section}

\DeclareMathOperator{\diag}{diag}

\def\cS{\mathcal S}
\def\cD{\mathcal D}
\def\cU{\mathcal U}

\def\cL{\mathcal L}

\begin{document}
\title[Semigroup automorphisms of total positivity]{Semigroup automorphisms of total positivity}

\author{Projesh Nath Choudhury}
\address[P.N.~Choudhury]{Department of Mathematics, Indian Institute of Technology Gandhinagar, Palaj, Gandhinagar
	Gujarat 382355, India}
\email{\tt projeshnc@iitgn.ac.in}

\author{Shaun Fallat}
\address[S. Fallat]{University of Regina, Regina, SK, Canada}
\email{\tt shaun.fallat@uregina.ca}

\author{Chi-Kwong Li}
\address[C.K. Li]{Department of Mathematics, College of William and Mary, Williamsburg, VA 23187-8793, USA}
\email{\tt cklixx@wm.edu}

\date{\today}

\begin{abstract}
Totally positive (TP) and totally nonnegative (TN) matrices connect to analysis, mechanics, and to dual canonical bases in reductive groups, by well-known works of Schoenberg, Gantmacher--Krein, Lusztig, and others. TP matrices form a multiplicatively closed semigroup, contained in the larger monoid of invertible totally nonnegative (ITN) matrices. Whitney and Berenstein--Fomin--Zelevinsky found bidiagonal factorizations  of all  $n\times n$ ITN and TP matrices into multiplicative generators; a natural question now is to classify the multiplicative automorphisms of these semigroups.	In this article, we classify all automorphisms of these semigroups of ITN and TP matrices. In particular, we show that the automorphisms are the same, and they respect the multiplicative generators.	

\end{abstract}



\maketitle

\section{Introduction and main results}

A real matrix $A$ is totally positive (TP) if all its minors are positive, and  totally nonnegative (TN) if all minors of $A$ are nonnegative. Total positivity is an evergreen area of mathematics with a rich history.  The pioneering works by Gantmacher--Krein \cite{gantmacher-krein} and Schoenberg \cite{S55}, the early monographs \cite{GK50, K68}, and the more recent works \cite{FZ02,fallat-john,GM96,pinkus} all provide extensive references to the intriguing history of total positivity, as well as tales of its many unforeseen applications in several branches of mathematics. These includes analysis, approximation theory, cluster algebras, combinatorics, differential equations, Gabor analysis, integrable systems, matrix theory, probability and statistics, and
representation theory
\cite{BFZ96,Bre95,fallat-john,FZ02,GRS18,K68,KW14,Lu94,pinkus,Ri03,Sch07}.

The notion of total positivity was first studied by Fekete--P\'olya \cite{FP12} in 1912 in terms of the  one sided P\'olya frequency sequences and their variation diminishing property -- which may be regarded as originating in the famous 1883 memoir of Laguerre \cite{Laguerre}. In the same paper, Fekete showed that TP matrices are characterized by the positivity of the contiguous minors which was later extended in 1955 to totally positivity of order $k$ by Schoenberg \cite{S55}. Subsequently, in 1937 Gantmacher--Krein \cite{gantmacher-krein} characterized TP (or TN)  matrices in terms of  the positivity (or non-negativity) of the spectra of all submatrices. In 1941, Gantmacher--Krein \cite{GK50} provided a fundamental characterization of total positivity in terms of variation diminution and very recently Choudhury \cite{C22} refined these results by establishing a single-vector test.

Given an integer $n \geq 1$, the Cauchy--Binet formula implies that the set of all invertible TN (ITN) matrices and the subset of all TP matrices are closed under multiplication. Denoting these by $ITN(n) \supsetneq TP(n)$ respectively, we see that $TP(n)$ is a semigroup of $GL_n(\mathbb{R})$ and that $ITN(n)$ is a submonoid of $GL_n(\mathbb{R})$. The following classical result of Cryer \cite{C1, C2}  shows that the study of $ITN(n)$ can be reduced to the investigation of its subsemigroup of upper triangular ITN matrices with one on the diagonal and subgroup of diagonal ITN matrices:

\begin{lemma}
	A matrix $A\in GL_n(\mathbb{R})$ is totally nonnegative if and only if $A$ has Gaussian decomposition $A=LDU$, where $L,U$ are lower and upper triangular ITN matrices with ones on the main diagonal and $D$ is a diagonal ITN matrix.
\end{lemma}

In 1952, Whitney \cite{Whitney} (a student of Schoenberg)  identified the infinitesimal generators of $ITN(n)$ and $TP(n)$ as the Chevalley generators of the corresponding Lie algebra. Whitney showed that  each ITN and TP matrix can be factorized using elementary bidiagonal matrices.  An $n\times n$ matrix of the form $I+w E_{ij}$ with $w \geq 0$, $I$ the identity matrix, and $E_{ij}$ an elementary matrix, is called an elementary bidiagonal matrix if $|i-j| = 1$. Elementary bidiagonal matrices are invertible totally nonnegative. We owe to Whitney the following factorization of invertible TN matrices.

\begin{theorem}[Whitney \cite{Whitney}]
	Every $n \times n$ ITN matrix $A$ is of the form
	
	\begin{equation}\label{bifact}
		A:=A({\bf w, w^{'}, d})=\prod_{j=1}^{n-1} \prod_{k=n-1}^{j} \big{(} I+w_{j,k}E_{k+1,k} \big{)}~D \prod_{j=n-1}^{1} \prod_{k=n-1}^{j} \big{(} I+ w'_{j,k+1}E_{k,k+1} \big{)},
	\end{equation}
	where  ${\bf w}=(w_{j,k})_{1 \leq j \leq k \leq n-1} \in [0,\infty)^{ n \choose 2},~ {\bf w'}=(w'_{j,k+1})_{1\leq j\leq k\leq n-1} \in [0,\infty)^{ n \choose 2}$; ${\bf d } \in \mathbb  (0,\infty)^n$ and $D=\diag({\bf d })$, and $D$ can be made to appear at either end of the factorization or between any of the elementary bidiagonal factors, and the factors are multiplied in the given order of $j$ and $k.$  Equivalently, the factorization can be arranged such that the product of upper bidiagonal factors appears on the left, while the product of lower bidiagonal factors appears on the right.
\end{theorem}

Building on this factorization, Lusztig \cite{Lu94} in 1994 extended the notion of total positivity to other semisimple Lie groups $G$, by defining the set $G_{\geq 0}$ of TN elements in $G$ as the semigroup generated by the (exponentiated) Chevalley generators. A few years later, using the factorization \eqref{bifact}, Berenstein--Fomin--Zelevinsky \cite{BFZ96} showed that $(0,\infty)^{n^2}$ is diffeomorphic to the space of $n\times n$ TP matrices via the map 
\[({\bf w, w^{'}, d})\xrightarrow{\psi} A({\bf w, w^{'}, d}).\]\smallskip

The primary goal of this article is to investigate semigroup automorphisms of total positivity. The automorphisms of various groups and semigroups of matrices were studied in  the literature \cite{ALMS06,CFL02}.  In the present work, we show that the automorphisms of both semigroups $ITN(n)$ and $TP(n)$ 
are equal, and involve a function purely of the determinant and diagonal/antidiagonal conjugation. Here is our first main result:

\begin{utheorem} \label{multiplicative}Let $n\geq 1$ be an integer and let $G(n) = \text{TP}(n)$ or $\text{ITN}(n)$. Then the following are equivalent for an arbitrary map $T : G(n) \to G(n)$.
	\begin{enumerate}
		\item[(i)] $T$ is a semigroup isomorphism.
		\item[(ii)] $T$ has the form 
		\[A \mapsto \mu(\det A)(\det A)^{-\frac{1}{n}} RAR^{-1},\]
		where $\mu: (0, \infty) \to (0, \infty)$ is a bijective  multiplicative map and 
		\[R=\begin{pmatrix}
			r_1 & \cdots & 0 \\
			\vdots & \ddots & \vdots \\
			0 & \cdots & r_n \\
		\end{pmatrix} \quad \text{or} \quad \begin{pmatrix}
		0 & \cdots & r_1 \\
		\vdots & \iddots & \vdots \\
		r_n & \cdots & 0 \\
		\end{pmatrix}\]
		for some positive real numbers $r_1, \dots, r_n$.
	\end{enumerate}
\end{utheorem}

\begin{rem}Here are two immediate consequences:
	\begin{enumerate}[(i)]
		\item Every semigroup automorphism of $ITN(n)$ sends each elementary bidiagonal generator to another such generator.
		\item In particular, every automorphism not only preserves the multiplicative structure of the semigroup of TP/ITN matrices, but it takes these ``multiplicative generators" onto themselves, and hence, every automorphism takes each bidiagonal factorization to another such.
	\end{enumerate}
\end{rem}
Our second main result shows that the automorphisms of both semigroups $ITN(n)$ and $TP(n)$ -- and in fact, of every multiplicative semigroup in between them satisfying some additional assumptions -- are equal.
\begin{utheorem}\label{multiplicative2} Let $n\geq 1$ be an integer and let $G(n)$ be a semigroup with $TP(n) \subseteq G(n) \subseteq ITN(n)$ such that $G(n)$ is closed under taking positive scalar multiples, and under conjugating by (anti-)diagonal matrices with positive (anti-)diagonal entries. Let $T : G(n) \to G(n)$ be an arbitrary map. Then the following are equivalent:
\begin{enumerate}
	\item[(i)] $T$ is a semigroup isomorphism with $T(TP(n))=TP(n)$.
	\item[(ii)] For all $A\in G(n)$,
	\[T(A) = \mu(\det A)(\det A)^{-\frac{1}{n}} RAR^{-1},\]
	where $\mu$ and $R$ are defined as in Theorem \ref{multiplicative} (ii).
\end{enumerate}
\end{utheorem}
\begin{example}
Note that for $n \geq 2$, the semigroup $G(n):=TP(n)\sqcup\{\diag(d_1,\ldots,d_n): d_1,\ldots,d_n>0\}$ strictly contains  $TP(n)$  and lies strictly within $ITN(n)$. By Corollary \ref{corpnc3.3}, any semigroup automorphism of $G(n)$ preserves  $TP(n)$. Thus Theorem \ref{multiplicative2} completely classifies the automorphisms of $G(n)$.
\end{example}

\begin{rem}
The multiplicative map $\mu: (0, \infty) \to (0, \infty)$ can be constructed as follows.  
Consider $\IR$ as an infinite-dimensional vector space over the rational
numbers. Let $\phi : \mathbb{R} \to \mathbb{R}$ be a $\mathbb{Q}$-linear isomorphism. Then $\mu(x) = \exp(\phi(\log(x))$ 
is a bijective multiplicative map on $(0,\infty)$, and there are no others. It is perhaps worth remarking here that the classical works of Banach \cite{Banach}, Sierpi\'nsky \cite{Sierpinsky}, and others show that $\mu$ is continuous if and only if it is measurable, if and only if $\phi$ is measurable; but there exist non-measurable $\mathbb{Q}$-linear maps $\phi$.
\end{rem}

We conclude our discussion of the main results with the following remark: The entrywise preservers of the set of TP matrices and of TN matrices were recently classified by Belton--Guillot--Khare--Putinar \cite{BGKP20}. However, this uncovered only the set-theoretic (entrywise) preservers. In contrast, the present work characterizes the matrix functions preserving  the (square) TP or ITN matrices, which form semigroups under multiplication. Thus the maps we classify preserve not only the TP/ITN matrices, but their respective multiplicative structure as well.

\medskip
\subsection*{Organization of the paper} The remaining sections of the paper are devoted to proving our main results. In Section 2, we show that a semigroup automorphism of total positivity can be extended to invertible total nonnegativity. In Section 3, we establish properties of some special subgroups and sub-semigroups  of matrices in $ITN(n)$ in terms of their centralizer. In Section 4, we prove Theorem \ref{multiplicative} for the $2\times 2$ case and in Section 5 for the general case.  We conclude the article with a quick proof of Theorem \ref{multiplicative2}.

\section{Reduction of the Main Problem}

The entirety of this paper (save a few lines to show Theorem B) is devoted to proving Theorem A. As the proof is rather technical and involved, we split it into three sections. We begin in this section by extending an automorphism of the semigroup of $n \times n$ TP matrices to a monoid automorphism of $n \times n$ invertible TN matrices. This requires three preliminary results.  The first one is a classical density result by Whitney.

\begin{theorem}[Whitney,~\cite{Whitney}]\label{dwhitney}
	Given an integer $n \geq 1$, the set of $n \times n$ $TP$
	matrices is dense in the set of $n \times n$ $TN$ matrices.
\end{theorem} 

We now recall a result of Karlin concerning invertible totally nonnegative matrices and positivity of their principal minors.
\begin{theorem}[Karlin,~\cite{K68}]\label{nontn}
If $A \in ITN(n)$, all principal minors of $A$ are positive.
\end{theorem}

The next preliminary result is the well known Cauchy--Binet determinantal identity. To state this identity, we introduce two pieces of notation.
\begin{defn} Let $m,n\geq 1$ be integers and $A\in \mathbb{R}^{m\times n}$.
	\begin{enumerate}[(i)]
		\item Define $[n]:=\{1,2,\ldots,n\}$.
		\item Given subsets $\alpha \subseteq  [m]$ and $\beta \subseteq  [n]$, define $A[\alpha|\beta]$ to be the submatrix of $A$ with rows indexed by $\alpha$ and the columns indexed by $\beta$.
	\end{enumerate}
\end{defn}

\begin{theorem}[Cauchy--Binet formula]\label{C-BF}
Given matrices $A\in \mathbb{R}^{m \times n}$ and $B\in \mathbb{R}^{n \times m}$, we have
\[\det(AB)=\sum_{\gamma\subseteq [n], ~|\gamma|=m} \det A\left[[m]|\gamma\right] \det B\left[\gamma|[m]\right].\]
\end{theorem}

We now provide a characterization of $ITN(n)$ in terms of the following sets:  Given an integer $n \geq 1$, define
\begin{eqnarray}
	\cS^L(n) &:= & \{ X \in \mathbb{R}^{n \times n} : XA \in TP(n)\ \forall A \in TP(n)\},\nonumber\\
\cS^R(n) &:= & \{ X \in \mathbb{R}^{n \times n} : AX \in TP(n)\ \forall A \in TP(n)\}, \\
\cS(n) &:= & \cS^L(n) \cap \cS^R(n).\nonumber
\end{eqnarray}
It is easy to verify that all of these are submonoids of $GL_n(\mathbb{R})$.

\begin{prop}\label{prop1}
The monoids $\cS^L(n), \cS^R(n), \cS(n)$ are all equal, and equal to $ITN(n)$, the monoid of all invertible $n \times n$ TN matrices.
\end{prop}

\begin{proof}
We first claim that $ITN(n)=\cS^R(n)$. To show  $ITN(n) \subseteq \cS^R(n)$, suppose $X$ is $ITN(n)$. Then for all $A \in TP(n)$, we may deduce that $AX$ is TP, by a simple application of the Cauchy--Binet identity and Theorem \ref{nontn}. Indeed, for any two nonempty  subsets $\alpha, \beta \subseteq [n]$ with the same cardinality, we have

\[ \det (AX)[\alpha|\beta] = \sum_{\gamma\subseteq [n], ~|\gamma|=|\alpha|} \det A[\alpha|\gamma] \det X[\gamma|\beta] \geq \det A[\alpha|\beta] \det X[\beta|\beta] >0. \]


For the reverse inclusion, suppose $X\in \cS^R(n)$, so that $X$ is invertible. By the Whitney density theorem \ref{dwhitney}, there exists a sequence $A_k \in TP(n)$ converging to $I_n$.  Now for any $\alpha, \beta \subset [n]$ with $|\alpha| = |\beta|>0$, we have
\[
\det (A_k X) [ \alpha | \beta ] \to \det X [\alpha | \beta].
\]
Since the left-hand sequence of determinants is positive, we get $\det X [\alpha | \beta] \geq 0$ for all $\alpha,\beta$ with $|\alpha| = |\beta|>0$. Hence $X \in ITN(n)$, and so $ITN(n)=\cS^R(n)$.

Similarly, one shows that $ITN(n)$ equals $\cS^L(n)$, and hence equals their intersection $\cS(n)$.
\end{proof}

For the rest of our discussion, we will use $\cS(n)$ to denote the monoid $ITN(n)$. We may drop the dependence on the parameter $n$ if it is clear from context. Using Proposition \ref{prop1}, we now show that any semigroup automorphism of $TP(n)$ can be extended to a monoid automorphism of $\cS(n)$.

\begin{prop}\label{prop2}
Let $T$ be a semigroup automorphism of $TP(n)$.  Then $T$ can be extended to a monoid automorphism $\hat{T}$ of $\cS(n)$.
\end{prop}
\begin{proof}
 Let $A\in TP(n)$ and define  $\hat T: \cS(n) \rightarrow \cS(n)$ by
\[\hat T(X): = T(A)^{-1}T(AX) = T(XA) T(A)^{-1} \quad \hbox{for all } X \in \cS(n).\]
Then $\hat T (I_n)=I_n$ and $\hat{T} \equiv T$ on $TP(n)$, since $T$ is an automorphism on $TP(n)$.
We now show that $\hat{T}$ is well defined and independent of the choice of $A$. Let  $X \in \cS(n)$. For any $B \in TP(n)$, by Proposition \ref{prop1}, $AX, BX, XA, XB, AXB \in TP(n)$. Thus $T(A)T(XB) = T(AXB) = T(AX)T(B)$.
It follows that
 \begin{equation}\label{eqprop2}
T(A)^{-1}T(AX) = T(XB)T(B)^{-1}~~ \hbox{for any}~~ B\in TP(n).
\end{equation}
If $A = B$ then $T(A)^{-1}T(AX) = T(XA)T(A)^{-1}=\hat T(X)$. By \eqref{eqprop2}, the definition of $\hat T$ is independent of the choice of $A \in TP(n)$.  Now to show $\hat T(X) \in \cS(n)$ for all $X\in \cS(n)$, it suffices to show that $\hat T(X)C \in TP(n)$ for all $C\in TP(n)$.  Let $C\in TP(n)$. Since $T$ is bijective, there exists $B \in TP(n)$ such that $C=T(B)$. By the definition of $\hat T$, $\hat T(X)C=T(XB)\in TP(n)$, since $T$ is an automorphism of $TP(n)$.  Thus $\hat T$ is a well-defined function from $\cS(n)$ to $\cS(n)$.

We next show that $\hat T$ is multiplicative on $\cS(n)$. For any $X_1, X_2 \in \cS(n)$ and $A\in TP(n)$,
\begin{eqnarray*}
\hat T(X_1X_2) &=& T(A)^{-1}T(AX_1X_2) \\&=& T(A)^{-1}T(AX_1X_2A)T(A)^{-1}\\
&=& T(A)^{-1}T(AX_1) T(X_2A)T(A)^{-1}\\& =& \hat T(X_1)\hat T(X_2).
\end{eqnarray*}
Thus $\hat T$ is multiplicative.  To show $\hat T$ is bijective, let $A\in TP(n)$ and denote $ T(A)=B$.  Define the map $U: \cS(n) \to \cS(n)$ by
	\[U(X) := (T^{-1}(B))^{-1} T^{-1}(BX)~~ \hbox{for all}~~ X \in \cS(n).\]
Then for any $X \in \cS(n)$, 
\begin{eqnarray*}
U\circ \hat T(X) &=& U(B^{-1}T(AX))\\
&=& (T^{-1}(B))^{-1} T^{-1}(B B^{-1}T(AX)) \\
&=& A^{-1} (AX)=X
\end{eqnarray*}
Similarly, one can show  $\hat T \circ U(X) = X$ for all $X\in \cS(n)$. Thus $U$ is the inverse of $\hat T$. 
\end{proof}

By Proposition \ref{prop2}, to prove $(i)\implies (ii)$ of Theorem \ref{multiplicative}, it suffices to consider $G(n)=\cS(n)$ and $T$ as a map from $\cS(n)$ to $\cS(n)$.

\section{Special ITN Matrices and their centralizer}
To study an automorphism $T$ of $\cS(n)$, it is useful to study some special submonoids of $\cS(n)$ that will be preserved by $T$. To begin with, we need some basic definitions and notations. These are isolated here and used below without further reference.
\begin{defn}
Let $m,n \geq 1$ be integers. 
\begin{enumerate}[(i)]
	\item We denote the $n\times n$ antidiagonal matrix with all antidiagonal entries $1$ by $P_n:$
	\[P_n:=\begin{pmatrix}
		0 & \cdots & 1 \\
		\vdots & \iddots & \vdots \\
		1 & \cdots & 0 \\
	\end{pmatrix} \in \{ 0, 1 \}^{n\times n}.\]
	\item Let $\cD(n)$  denote the set of all $n\times n$ diagonal matrices with positive diagonal entries.
	\item We denote by $Z^+(n) := Z(GL_n(\mathbb{R})) \cap ((0,\infty))^n$ the subgroup of scalar matrices with \textit{positive} diagonal entries.
	\item Define $\cL(n):=\{aI+bE_{i,{i-1}}\in \mathbb{R}^{n\times n}: a, b>0 \hbox{ and } 2\leq i\leq n\}$ 
	and $\cU(n):=\{aI+bE_{i,{i+1}}\in \mathbb{R}^{n\times n}: a, b>0 \hbox{ and } 1\leq i\leq n-1\}$.
	 \item For positive integers $n_1,\ldots,n_k$, define $\cS(n_1)\oplus\cdots\oplus \cS(n_k):=\{A_1\oplus\cdots \oplus A_k: A_i \in \cS(n_i) \text{ for all } i\in [k-1]\}$.
	 \item Let $\cS_0(n):=SL_n(\mathbb{R})\cap\cS(n)$. Similarly, we define the subsets $\cD_0(n), \cL_0(n)$ and $\cU_0(n)$.
\end{enumerate}
\end{defn}
We first identify the maximal subgroup contained in $\cS(n)$.

\begin{lemma} \label{groupd}
	Let $n\geq 2$. The set $\cD(n)$ is the unique maximal subgroup in $\cS(n)$.
\end{lemma}

\begin{proof}
Note that $\cD(n)$ is a group in $\cS(n)$.  Now, suppose $H$ is a group under multiplication in $\cS(n)$. Let $A\in H$ be a non-diagonal matrix.  Since $A$ is ITN, by Theorem \ref{nontn}, all the diagonal entries of $A$ are positive. Thus $A^{-1} \in H$ has a negative entry, a contradiction. Thus $H\subseteq \cD(n)$.
\end{proof}
An immediate consequence of the above lemma is the following:
\begin{cor}\label{corpnc3.3}
	If $G(n) = TP(n) \sqcup \mathcal{D}(n)$, and $\tilde{T} : G(n) \to G(n)$ is a multiplicative semigroup automorphism, then $\tilde{T}$ maps $TP(n)$ onto itself.
\end{cor}
\begin{proof}
	Since $\tilde{T}(I)=I$, $\tilde{T}(\cD(n))$ is a subgroup in $G(n)$. By Lemma \ref{groupd}, $\tilde{T}(\cD(n))\subseteq \cD(n)$. Since $\tilde{T}$ is bijective,  similarly one can show that $\tilde{T}^{-1}(\cD(n))\subseteq \cD(n)$,  and  so $\tilde{T}(\cD(n))= \cD(n)$. This implies $\tilde{T}(TP(n))=\tilde T \left(G(n)\setminus \cD(n)\right)=G(n)\setminus \tilde T(\cD(n))= TP(n)$.
\end{proof}

To determine some special submonoids of $\cS(n)$ on which $T$ will induce an automorphism, we now study matrices $D\in \cD(n)$ and their centralizers.

\begin{defn}
For $A \in \cS(n)$, the centralizer of $A$ with respect to $\cS(n)$ is denoted and defined as
\[C(A):= \{X \in \cS(n): AX = XA\}.\]

\end{defn}
Note that if $T$ is a semigroup automorphism of $\cS(n)$, then $T$ will map $C(A)$ onto $C(T(A))$. This leads us to the following properties: 

\begin{rem}\label{rem1}
	Let $n\geq 1$ be an integer and $T$ be an automorphism on $\cS(n)$. We have the following consequences:
	\begin{enumerate}[(i)]
		\item By a straightforward calculation, one can verify that $A \in \cS(n)$ satisfies $C(A) = \cS(n)$ if and only if $A = aI_n$ for some $a > 0$. Thus $T$ induces an automorphism on the subgroup $Z^+(n) = \{aI: a> 0\}$ of $\cS(n)$.
		\item $T(aI_n) = \gamma(a)I_n$ for some bijective multiplicative map $\gamma: (0, \infty) \rightarrow (0,\infty)$.
		\item Since $T(I)=I$, by Lemma \ref{groupd}, $T(\cD(n))=\cD(n)$. Thus $T$ induces an automorphism on $\cD(n)$.
	\end{enumerate}
\end{rem}

\begin{lemma} \label{CDk-0}   Let $n, n_1,\ldots,n_k$ be positive integers. Suppose 
	$D = d_1 I_{n_1} \oplus \cdots \oplus d_k I_{n_k} \in \cD(n)$ with $d_j \ne d_{j+1}$ for 
	$j \in[k-1]$. Then
	$C(D) = \cS(n_1) \oplus \cdots \oplus \cS(n_k)$.
\end{lemma}

\begin{proof}
Clearly $\cS(n_1) \oplus \cdots \oplus \cS(n_k)\subseteq C(D)$. For the converse, let $X\in C(D)$. Then $DX = XD$ and $X$ is ITN. Note that $DX = XD$ if and only if 
$X = (X_{ij})$ with $X_{jj} \in \mathbb{R}^{n_j \times n_j}$ for $j \in [k]$ and
$X_{ij} = 0$ whenever $d_i \neq d_j$. Furthermore, since $X$ is ITN, $X_{ij} = 0$ if $d_i = d_j$ but $|i-j|> 1$. Hence $C(D) = \cS(n_1) \oplus \cdots \oplus \cS(n_k)$.
\end{proof}

We now present an outline of the proof strategy for our first main result.

\begin{proof}[Proof-sketch for Theorem \ref{multiplicative}]
By the Cauchy--Binet formula, $TP(n)$ and $ITN(n)$ are closed under multiplication by positive scalars and under conjugation by (anti-)diagonal matrices with positive (anti-)diagonal entries. Thus the implication $(ii) \implies (i)$ is immediate. To show  $(i) \implies (ii)$, Proposition \ref{prop2}, allows us to restrict our attention to $\cS(n)$, and we then proceed by induction on $n$.  We prove the base case $n=2$ in the next section, and using this we complete the proof in Section 5.
\end{proof}
\section{Proof of Theorem \ref{multiplicative} for $n=2$}
In this section, we prove $(i)\implies (ii)$ of Theorem \ref{multiplicative} for $2\times2$ matrices. We begin with a few lemmas that play a crucial role in establishing the main result. As a first step, we characterize the elements of the semigroups $\cU(2)$ and $\cL(2)$.
\begin{lemma}\label{pc4.1}
	Let $A \in \cS(2)$ be a non-diagonal matrix. Then $DAD^{-1} \in C(A)$ for all $D\in \cD(2)$ if and only if   $A =a I_2 + bE_{12}$  or 
	$A =a I_2 + bE_{21}$ for some $a, b >0$.
\end{lemma}
\begin{proof}
	For the sufficiency, note that 
	$C(A) = \{ x I_2 + yE_{12}: x>0,~ y \ge 0\}$ if $A = a I_2 + b E_{12}$  and $C (A) = \{x I_2 + yE_{21}: x>0,~ y \ge 0\}$ if $A = aI_2 + bE_{21}$. Thus $DAD^{-1}\in C(A)$  for all $D\in \cD(2)$.   To show the necessity, suppose $A=\begin{pmatrix}
		a_{11} & 	a_{12} \\ 	a_{21} & 	a_{22}
	\end{pmatrix} \in \cS(2)$ is a non-diagonal matrix such that $DAD^{-1} \in C(A)$ for all $D\in \cD(2)$. We now claim that $a_{11}=	a_{22}$ and either $ a_{12}=0$ or 	$a_{21}=0$. Let $D=\diag(d_1,d_2)\in \cD(2)$ with $d_1\neq d_2$.  Since $A DAD^{-1}=DAD^{-1}A$, we have
	\[\begin{pmatrix}
		a^2_{11}+\frac{d_2}{d_1}a_{12} a_{21} & a_{12}a_{22}+\frac{d_1}{d_2} a_{11} a_{12}\\ a_{11} a_{21} +\frac{d_2}{d_1}a_{21}a_{22} &  a^2_{22}+\frac{d_1}{d_2}a_{12} a_{21}
	\end{pmatrix}=\begin{pmatrix}
		a^2_{11}+\frac{d_1}{d_2}a_{12} a_{21} & a_{11} a_{12}+  \frac{d_1}{d_2}a_{12}a_{22}\\  a_{21}a_{22}+\frac{d_2}{d_1} a_{11} a_{21} &  a^2_{22}+\frac{d_2}{d_1}a_{12} a_{21}
	\end{pmatrix}.\]
Comparing the $(1,1)$ entries of the above matrices, we have either $a_{12}=0$ or $a_{21}=0$. If $a_{12}\neq 0$, comparing the $(1,2)$ entries of the above matrices, we have $a_{11}=a_{22}$. Otherwise comparing the $(2,1)$ entries of the above matrices, we have $a_{11}=a_{22}$. Thus  $A =a I_2 + bE_{12}$  or 
	$A =a I_2 + bE_{21}$ for some $a, b >0$.
\end{proof}

Using the above result, we now show that if $T$ is an automorphism on $\mathcal S(2)$, then either $T$, or the map $X \mapsto P_2 T(X) P_2$, induces an automorphism on the semigroups $\mathcal U(2)$ and $\mathcal L(2)$.

\begin{lemma}\label{pc4.2}
	Let $T: \cS(2) \to \cS(2)$ be a semigroup isomorphism. Then $ T(\cU(2)) = \cU(2)$ and $ T(\cL(2))=\cL(2)$ or $ T(\cU(2)) = \cL(2)$ and $ T(\cL(2))=\cU(2)$. 
\end{lemma}
\begin{proof}
	Let $A\in \cU(2)$. We first show that $ T(A)\in \cU(2)$ or $\cL(2)$ using Lemma \ref{pc4.1}. Let $D\in \cD(2)$. Since $T(\cD(2))=\cD(2)$, $D=T(\hat D)$ for some $\hat D\in \cD(2)$ and $T(A)$ is a non-diagonal matrix. Thus $DT(A)D^{-1}=T(\hat D A {\hat D}^{-1})\in T(C(A))$,  by Lemma \ref{pc4.1}. Since $ T(C(A))=C(T(A))$,  $T(A)\in \cU(2)$ or $\cL(2)$. Similarly, we can show that  if $A\in \cL(2)$ then $T(A)\in \cU(2)$ or $\cL(2)$. Since $T$ is an automorphism on $S(2)$, $T^{-1}(A)\in \cU(2)$ or $\cL(2)$, if $A\in \cU(2)$ or $\cL(2)$. 
	
	Let $A\in \cU(2)$ and without loss of generality assume that $T(A)\in \cU(2)$. We next claim that $T(\cU(2))= \cU(2)$ and $T(\cL(2))= \cL(2)$. If $T(\cU(2)) \nsubseteq \cU(2)$, then there exists $B\in \cU(2)$ such that $T(B)\in \cL(2)$. This implies $AB \in \cU(2)$ but $T(AB)=T(A) T(B)$ is not triangular, a contradiction. Thus $ T(\cU(2)) \subseteq \cU(2)$.  Similarly, if $B \in \cL(2)$ and $T(B)\in \cU(2)$ then $T(AB)\in \cU(2)$ but $AB$ is not triangular, a contradiction. Thus  $T(\cL(2))\subseteq \cL(2)$. Again, by a similar argument one can conclude that $ T^{-1}(\cU(2)) \subseteq \cU(2)$ and  $T^{-1}(\cL(2))\subseteq \cL(2)$. Hence  $T(\cU(2))= \cU(2)$ and $T(\cL(2))= \cL(2)$. This completes the proof.
\end{proof}

Let $T: \cS(2) \to \cS(2)$ be a semigroup isomorphism. By the above lemma, we may replace $T$ with the map $X \mapsto P_2 T(X) P_2$ if needed, and henceforth \textbf{assume} that
\begin{equation}\label{eqwlog2}
T(\cU(2)) = \cU(2)\quad \text{and} \quad  T(\cL(2))=\cL(2).
\end{equation}
 We next show that $T$ induces an  automorphism on the semigroups $\cU_0(2)$ and $\cL_0(2)$.
\begin{lemma}\label{pc4.3}
	Let $T: \cS(2) \to \cS(2)$ be a semigroup isomorphism satisfying \eqref{eqwlog2}. Then $ T(\cU_0(2)) = \cU_0(2)$ and $ T(\cL_0(2))=\cL_0(2)$. 
\end{lemma}
\begin{proof}
	By \eqref{eqwlog2}, $T(\cU(2)) = \cU(2)$ and $ T(\cL(2))=\cL(2)$. 
	 To show $T(\cU_0(2)) = \cU_0(2)$, let $x>0$. Then $A= \begin{pmatrix}
	 	1 & x \\ 0 & 1
	 \end{pmatrix}\in \cU_0(2)$ and $T(A)= \begin{pmatrix}
		a & a_{12} \\ 0 & a
	\end{pmatrix}$ for some $a, a_{12}>0$, since $\cU_0(2) \subset \cU(2)$.  We claim that $a=1$. Note that
	\begin{equation}\label{eq1}
		 \begin{pmatrix}
			1 & x \\ 0 & 1
		\end{pmatrix} \begin{pmatrix}
		1 & x \\ 0 & 1
		\end{pmatrix}=  \begin{pmatrix}
		1 & 2x \\ 0 & 1
		\end{pmatrix} = \begin{pmatrix}
			1 & 0 \\ 0 & 1/2
		\end{pmatrix} \begin{pmatrix}
		1 & x \\ 0 & 1
		\end{pmatrix}\begin{pmatrix}
			1 & 0 \\ 0 & 2
		\end{pmatrix}
	\end{equation}
	Since $T$ induces an automorphism on the group $\cD(n)$, $T\left(\begin{pmatrix}
		1 & 0 \\ 0 & 2
	\end{pmatrix}\right)=\begin{pmatrix}
		d_1 & 0 \\ 0 & d_2
	\end{pmatrix}$ for some $d_1,d_2>0$.  Applying $T$ on both sides of \eqref{eq1}, we have 
	\[\begin{pmatrix}
		a^2 & 2a a_{12} \\ 0 & a^2
	\end{pmatrix}=\begin{pmatrix}
		a & \frac{d_2}{d_1}a_{12} \\ 0 & a
	\end{pmatrix}.\] Thus $a=1$ and $T(\cU_0(2)) \subseteq \cU_0(2)$.  Since $T$ bijective, by an analogous argument we can show that  $T^{-1}(\cU_0(2)) \subseteq \cU_0(2)$. Thus  $T(\cU_0(2)) = \cU_0(2)$.
	
	Similarly, one can show that $T(\cL_0(2)) = \cL_0(2)$.  This concludes the proof.
\end{proof}

Using the above lemmas, we are now ready to prove the implication $(i)\implies (ii)$ of Theorem \ref{multiplicative} in the $2\times2$ case.

\begin{proof}[Proof of Theorem \ref{multiplicative}$(i)\implies (ii)$ for the $2\times2$ case]
Replacing $T$ by the map $X \mapsto P_2 T(X) P_2$ if necessary, we may assume that $T$ satisfies \eqref{eqwlog2}.	By Lemma \ref{pc4.3}, we have $ T(\cU_0(2)) = \cU_0(2)$ and $ T(\cL_0(2)) = \cL_0(2)$. Then $ T(I + xE_{12}) = I + f(x)E_{12}$ for all $x >0$, where $f$ is a bijective map on $(0,\infty)$. We may assume $f(1)=1$, otherwise replace $ T$ by the map $A \mapsto \begin{pmatrix}
		\frac{1}{ f(1)} & 0 \\ 0 & 1
	\end{pmatrix}T(A) \begin{pmatrix}
		f(1) & 0 \\ 0 & 1
	\end{pmatrix}$.  We now show that $f$ is additive on $(0,\infty)$. Let $a_1,a_2>0$. Using the identity $ \begin{pmatrix}
	1 & a_1+a_2 \\ 0 & 1
	\end{pmatrix}= \begin{pmatrix}
	1 & a_1 \\ 0 & 1
	\end{pmatrix}  \begin{pmatrix}
	1 & a_2 \\ 0 & 1
	\end{pmatrix}$, we have
	\begin{eqnarray*}
		T(I+(a_1+a_2)E_{12}) &=&(I+f(a_1)E_{12})(I+f(a_2)E_{12}) \\
		&=& I + (f(a_1)+f(a_2)) E_{12}.
	\end{eqnarray*}
	Thus $f(a_1+a_2)=f(a_1)+f(a_2)$ for all $a_1,a_2>0$.  We next show that $f$ is multiplicative on $(0,\infty)$.
	Since $T(\cD(2))=\cD(2)$,
	
	\[T\left(\begin{pmatrix}
		x & 0 \\ 0 & 1/x
	\end{pmatrix}\right)=\begin{pmatrix}
		h_1(x) & 0 \\ 0 & h_2(x)
	\end{pmatrix},\]
	for all $x>0$. Then $h_1$ and $h_2$ are multiplicative on $(0,\infty)$. Note that 
	
	\begin{equation}\label{eq2}
		\begin{pmatrix}
			x & 0 \\ 0 & 1/x
		\end{pmatrix} \begin{pmatrix}
			1 & 1 \\ 0 & 1
		\end{pmatrix}\begin{pmatrix}
			1/x & 0 \\ 0 & x
		\end{pmatrix}= \begin{pmatrix}
			1 & x^2 \\ 0 & 1
		\end{pmatrix}.
	\end{equation}
Applying $T$ on both sides of \eqref{eq2}, we have 
	\[\begin{pmatrix}
		1 & h_1(x)/h_2(x) \\ 0 & 1
	\end{pmatrix}=\begin{pmatrix}
		1 & f(x^2) \\ 0 & 1
	\end{pmatrix}.\]
	Thus 
	\begin{equation}\label{eqnf}
		f(x)=\sqrt{\frac{h_1(x)}{h_2(x)}} \hbox{  for all  } x>0.
	\end{equation}Since $h_1$ and $h_2$ are multiplicative on $(0,\infty)$, $f$ is also multiplicative on $(0,\infty)$.  Thus $f$ is bijective, additive and multiplicative on $(0,\infty)$, and so can be extended to  a field automorphism on $\IR$ preserving positive numbers. Hence $f(x) = x$ for all $x > 0$.
	
We next show that $h_1(x)=x$ and $h_2(x)=\frac{1}{x}$ for all $x>0$. Let $x>1$.  Again, note that 
	\begin{equation}\label{eq3}
		\begin{pmatrix}
			1 & 0 \\ 1 & 1
		\end{pmatrix} \begin{pmatrix}
			1 & x^2-1 \\ 0 & 1
		\end{pmatrix}\begin{pmatrix}
			x & 0 \\ 0 & 1/x
		\end{pmatrix}= 
		\begin{pmatrix}
			1/x & 0 \\ 0 & x
		\end{pmatrix} \begin{pmatrix}
			1 & x^2-1 \\ 0 & 1
		\end{pmatrix}\begin{pmatrix}
			1 & 0 \\ 1 & 1
		\end{pmatrix}.
	\end{equation}
	Since $ T(\cL(2)) = \cL(2)$, $ T(I + aE_{21}) = I + g(a)E_{21}$ for all $a >0$, where $g$ is a bijective map on $(0,\infty)$.  Applying $T$ on both sides of \eqref{eq3}, we have 
	
	\begin{equation}
		\begin{pmatrix}
			1 & 0 \\ g(1) & 1
		\end{pmatrix} \begin{pmatrix}
			1 & x^2-1 \\ 0 & 1
		\end{pmatrix}\begin{pmatrix}
			h_1(x) & 0 \\ 0 & h_2(x)
		\end{pmatrix}= 
		\begin{pmatrix}
			\frac{1}{h_1(x)} & 0 \\ 0 & \frac{1}{h_2(x)}
		\end{pmatrix} \begin{pmatrix}
			1 & x^2-1 \\ 0 & 1
		\end{pmatrix}\begin{pmatrix}
			1 & 0 \\ g(1) & 1
		\end{pmatrix}.
	\end{equation}
	
	Thus
	\begin{equation}\label{eq4}
		\begin{pmatrix}
			h_1(x) & (x^2-1)h_2(x)\\h_1(x)g(1) & (x^2-1)h_2(x)g(1)+h_2(x)
		\end{pmatrix}= \begin{pmatrix}
			\frac{1}{h_1(x)}+\frac{(x^2-1)g(1)}{h_1(x)} & \frac{(x^2-1)}{h_1(x)} \\\frac{g(1)}{h_2(x)} & \frac{1}{h_2(x)}
		\end{pmatrix}.
	\end{equation}
Since $g(1)>0$, comparing the $(2,1)$ entries of the above matrices, we have $h_1(x)h_2(x)=1$ for all $x>1$.  If $x<1$, we can do the above analysis for $y=1/x$ and conclude that $h_1(y)h_2(y)=1$. Since $h_1$ and $h_2$ are multiplicative, $\frac{1}{h_1(x)h_2(x)}=1$. Thus $h_1(x)h_2(x)=1$ for all $x>0$.  From \eqref{eqnf}, we have $h_1(x)=f(x)=x$ and $h_2(x)=\frac{1}{x}$ for all $x>0$.  
	
Next we show that $g(x)=x$ for all $x>0$. Taking $x=2$ and looking at the first entries on both sides of \eqref{eq4}, we have $g(1)=1$. Thus $T(I+E_{21})=I+E_{21}$. Now by a similar argument as for $f$, we can show that $g$ is additive, multiplicative and bijective on $(0,\infty)$. Thus $g(x)=x$ for all $x>0$. This implies that for all $x>0$,

\begin{equation}\label{pnceq2by2}
 T\left(\begin{pmatrix}
	1 & x \\ 0 & 1
\end{pmatrix}\right)=\begin{pmatrix}
	1 & x \\ 0 & 1
\end{pmatrix},\quad T\left(\begin{pmatrix}
1 & 0 \\ x & 1
\end{pmatrix}\right)=\begin{pmatrix}
1 & 0 \\ x & 1
\end{pmatrix},\quad T\left(\begin{pmatrix}
x & 0 \\ 0 &\frac{1}{x}
\end{pmatrix}\right)=\begin{pmatrix}
x & 0 \\ 0 & \frac{1}{x}
\end{pmatrix}. \end{equation}
By the Whitney bidiagonal factorization, every $A\in\cS_0(2)$ can be written as $A=LDU$, where $L\in \cL_0(2)$, $U\in \cU_0(2)$, and $D\in \cD_0(2)$.  Hence, using \eqref{pnceq2by2}, we have
\[T(A)=A \text{ for all } A \in \cS_0(2).\]
 Notice that we have used the maps $X \mapsto P_2T(X)P_2$ and $X \mapsto DT(A)D^{-1}$ to conclude that $T(A)=A$ for all $A \in \cS_0(2)$, where $D$ is a diagonal matrix with positive diagonal entries. Moreover, all these maps preserve the determinant.

 Now, let $A \in \cS(2)$ be arbitrary. Then $A$ can be written as $A = ((\det A)^{\frac{1}{2}} I_2)A_0$ with $\det A_0 = 1$. Thus
	\begin{eqnarray*}
		T(A) &=& T((\det A)^{\frac{1}{2}} A_0)\\
		&=&= \gamma((\det A)^{\frac{1}{2}})  T(A_0)\\
		&=&= \gamma((\det A)^{\frac{1}{2}})  A_0\\
		&=&\mu(\det A)(\det A)^{-\frac{1}{2}} A ,
	\end{eqnarray*}
where $\mu(x):=\sqrt{\gamma(x)}$ is a bijective multiplicative map on $(0,\infty)$. Hence  for all $A\in\cS(2)$,
\[T(A)=\mu(\det A)(\det A)^{-\frac{1}{2}} RAR^{-1}\]
for some bijective  multiplicative map $\mu:(0,\infty) \to (0,\infty)$ and $R = \begin{pmatrix}
	r_1 & 0 \\ 0 & r_2
\end{pmatrix}$ or $\begin{pmatrix}
0 & r_1 \\ r_2 & 0
\end{pmatrix}$ with $r_1,r_2>0$. This concludes the proof for the $2 \times 2$ case.
\end{proof}

\section{Proof of Theorem \ref{multiplicative} for $n\geq3$}
In this section, we conclude the proof of Theorem \ref{multiplicative}  and also prove Theorem \ref{multiplicative2}.  Throughout this section we assume $n\geq 3$, unless stated otherwise. To apply the induction hypothesis, we  first show that if $T:\cS(n)\to \cS(n)$ is an isomorphism then either $T$ or the map $X \mapsto P_nT(X)P_n$ induces automorphisms on several submonoids and subsemigroups of $\cS(n)$. To continue our discussion, we set some basic notations.

\begin{notation}
	Given integers $1\leq k \leq n-1$, define the diagonal matrix $D_k$ to have $(j,j)$ entry $j$ if $j \neq k+1$, and $(k+1,k+1)$ entry $k$.
\end{notation}
\begin{rem}
More generally, in the rest of this section, one can work with any diagonal matrices $D_1, \dots, D_{n-1}$, where each $D_k$ has diagonal entries $d_{k,j} > 0$ such that $d_{k,k} = d_{k,k+1}$ but all other pairs of consecutive diagonal entries in $D_k$ are unequal.
\end{rem}


By Lemma \ref{CDk-0}, the centralizer of $D_k$ with respect to $\cS(n)$ is $C(D_k) = \{ F_1 \oplus X \oplus F_2 \in \cS(n): 
F_1\in \cD(k-1), F_2\in \cD(n-k-1), X \in \cS(2)\}$ for $k\in [n-1]$. We next provide a characterization of the  submonoid  $C(D_k)$ for all $k\in[n-1]$.
	
\begin{lemma}\label{CDk-01} Let $D \in \cD_n$ such that $C(D)\neq \cD(n)$. Then $C(D) = C(D_k)$ for some $k \in[n-1]$ if and only if  there is no $\hat D \in \cD(n)$ such that 
	$\cD(n) \subsetneq C(\hat D) \subsetneq C(D)$. Moreover, if $T:\cS(n)\to \cS(n)$ is an isomorphism and $i \in [n-1]$, then $T(C(D_i)) = C(D_j)$ for some $j\in [n-1]$.
\end{lemma}

\begin{proof}
	Suppose $D \in \cD_n$ such that $C(D) = C(D_k)= \{F_1 \oplus X \oplus F_2: 
	F_1\in \cD(k-1), F_2\in \cD(n-k-1), X \in \cS(2)\}$ for some $k \in [n-1]$.  Let $\hat D=\diag( \hat{d_1},\ldots, \hat{d_n}) \in \cD(n)$. If $C(\hat D)$  is a proper subset of $C(D)$ then $\hat{d}_i\neq \hat{d}_{i+1}$ for all $i \in [n-1]$ and hence $C(\hat D)= \cD(n)$. Thus  there is no $\hat D \in \cD(n)$ such that $\cD(n) \subsetneq C(\hat D) \subsetneq C(D)$.
	
	For the converse, let $D=\diag(d_1,\ldots,d_n) \in \cD_n$  and assume that  there is no $\hat D \in \cD(n)$ such that 
	$\cD(n) \subsetneq C(\hat D) \subsetneq C(D)$.  Since $C(D)\neq \cD(n)$, there exists 
	$k \in [n-1]$ such that $d_k=d_{k+1}$. Moreover, the fact that there is no $\hat D \in \cD(n)$ with
	$\cD(n) \subsetneq C(\hat D) \subsetneq C(D)$ ensures that no other $i \in [n-1]$ satisfies $d_i=d_{i+1}$. Thus  $C(D) = C(D_k)$.

	For the second part, fix $1\leq i\leq n-1$ and  claim that  $T(C(D_i)) = C(D_j)$ for some $j \in[n-1]$.  If not, then there exists $\hat D \in \cD(n)$ such that $\cD(n) \subsetneq C(\hat D) \subsetneq T(C(D_i))$, since $T(C(D_i))\neq \cD(n)$ by Remark \ref{rem1}. Since $T$ is an automorphism on $\cS(n)$, $C\left(T^{-1}(\hat D)\right)=T^{-1}(C(\hat D))$ and by Remark \ref{rem1}, $T(\cD(n)) =  \cD(n)$. This implies that $\cD(n) \subsetneq C\left(T^{-1}(\hat D)\right)  \subsetneq C(D_i)$, a contradiction by the first part. Thus  $T(C(D_i)) = C(D_j)$ for some  $j \in [n-1]$.
\end{proof}

By the preceding lemma, any automorphism $T$ of $\cS(n)$ maps $C(D_i)$ onto  $C(D_j)$ for some $j$. We now show that this correspondence is highly constrained. Specifically, we show that $T(C(D_i))$ is either $C(D_i)$ or $C(D_{n-i})$ for all $i \in [n-1]$. To proceed, we introduce a semigroup of  $C(D_i)$ that underpins the subsequent argument.  For $k\in [n-1]$, define the set
\begin{equation}\label{sgrpC_k}
	C_k:=\{aI_{k-1} \oplus X \oplus bI_{n-k-1} \in \cS(n): a,b > 0,  X \in \cS(2)\setminus\cD(2)\}.\end{equation}
It is easy to verify that $C_k$ is a semigroup in $\cS(n)$ for all $k\in[n-1]$. We now list a basic property of the elements of $C_i$ for all $i\in[n-1]$. The proof is omitted, as it is a straightforward verification.

\begin{lemma}\label{pc-3}
	Let $n \ge 4$.  Then all matrices in $C_i$ and $C_j$ commute if and only if $|i-j| > 1$. If instead $|i-j| =1$, not all matrices in $C_i$ and $C_j$ commute.
\end{lemma}
 We now provide a characterization of the matrices in $C_k$ which will aid in the classification of $T(C_k)$. 
\begin{lemma}\label{pc-05}
	Let  $X \in C(D_k)$  be a non-diagonal matrix.  Then $X$ is not of the form 
	$aI_{k-1} \oplus Z \oplus bI_{n-k-1}$ if and only if  $C(X)$ is a proper subset of $C(Y)$ for some non-diagonal $Y\in C(D_k)$.
\end{lemma}
\begin{proof}
	If $X = A_1 \oplus A_2 \oplus A_3$ is not of the said form,
	then either $A_1$ or $A_3$ is not a scalar matrix. Consider
	$Y = aI_{k-1} \oplus A_2 \oplus bI_{n-k-1}$ for some $a, b > 0$. Then $Y\in C(D_k)$ and 
	$C(X)$ is a proper subset of $C(Y)$.
	
	Conversely, suppose $X = aI_{k-1} \oplus Z \oplus bI_{n-k-1}$. Then 
	$C(X) = \cS(k-1) \oplus C(Z) \oplus \cS(n-k-1)$.
	Let $Y = B_1 \oplus W \oplus B_3 \in C(D_k)$ be non-diagonal and $C(X) \subseteq C(Y)$. Then  $B_1=cI_{k-1}, B_3=dI_{n-k-1}$ for some $c,d>0$ and $C(Z) \subseteq C(W)$. Since $Z$ and $W$ are $2 \times 2$ non-scalar  matrices and $Z \in C(W)$, $Z=\alpha I_2+\beta W$ for some $\beta>0$ and $\alpha\in \mathbb{R}$. This implies that $C(Z)=C(W)$. Thus $C(X)$ is not a proper subset of $C(Y)$.
\end{proof}
~~~~~~~~~~~~~~

Using the above lemma, we next show that $T(C_j)=C_j$ or $C_{n-j}$.

\begin{lemma} \label{pc-4}
	Let $n \ge 3$ and  $T:\cS(n)\to \cS(n)$ be a semigroup isomorphism.  Then either
	\[ T(C_j) = C_j \text{ for all } 1 \le j < n, \text{ or } T(C_j) =  C_{n-j} \text{ for all }1 \le j < n.\]
\end{lemma}
\begin{proof}
 We first claim that for each $j\in[n-1]$ there exists $k\in[n-1]$ such that $T(C_j) = C_k$.  Fix $1\leq j\leq n-1$. By Lemma \ref{CDk-01}, $T(C(D_j))=C(D_k)$ for some $k\in[n-1]$. Since $C_j \subsetneq C(D_j)$, $T(C_j)\subseteq C(D_k)$.  Let $X = aI_{j-1} \oplus Z \oplus bI_{n-j-1}\in C_j$ and claim that $T(X)\in C_k$. If not, then there exists non-diagonal $Y\in C(D_k)$ such that $C(T(X))$ is a proper subset of $C(Y)$. Since $T$ is invertible and $T(\cD(n))=\cD(n)$, $T^{-1}(Y)\in C(D_j)$ is non-diagonal and $C(X)\subsetneq C(T^{-1}(Y))$, a contradiction by Lemma \ref{pc-05}. Thus $T(C_j)\subseteq C_k$. Since $T$ is bijective, it follows using similar reasoning that $T^{-1}(C_k)\subseteq C_j$. Thus  $T(C_j)= C_k$, where $ k\in[n-1]$.

Given the above, the result is trivial for $n=3$. Assume henceforth that $n\geq4$. Since $T$ is bijective, we have $(T(C_1), \dots, T(C_{n-1}))$
is a permutation of $(C_1, \dots, C_{n-1})$. We now claim that $T(C_1)=C_1$ or $T(C_1)=C_{n-1}$. By Lemma \ref{pc-3}, the matrices in $T(C_1)$ commute with all matrices in the $n-3$ sets $T(C_3), \dots, T(C_{n-1})$, and the matrices in $T(C_{n-1})$ commute with all matrices in the $n-3$ sets $T(C_1), \dots, T(C_{n-3})$. Notice that $C_1$ and $C_{n-1}$ are the only two sets whose all elements commute with all matrices from exactly $n-3$ of the sets  $C_1, \dots, C_{n-1}$. Thus $T(C_1)=C_1$ or $C_{n-1}$ and $T(C_{n-1})= C_{n-1}$ or $C_1$.

Assume that $T(C_1)=C_1$. Then $T(C_{n-1})= C_{n-1}$. To complete the proof, it remains to show that $T(C_j)=C_j$ for $2\leq j\leq n-2$. By Lemma \ref{pc-3}, $C_2$ is the unique set among $C_1, \dots, C_{n-1}$ whose elements do not commute with all elements of $C_1$, and similarly $C_{n-2}$ is the unique set whose elements do not commute with all elements of $C_{n-1}$. This implies that not all matrices in  $C_1$ and $T(C_2)$ commute, and not all matrices in  $C_{n-1}$ and $T(C_{n-2})$ commute. Thus $T(C_2)=C_2$ and $T(C_{n-2})=C_{n-2}$. Proceeding similarly, we can show that  $T(C_j)=C_j$ and $T(C_{n-j})= C_{n-j}$ for $3\leq j \leq \lfloor \frac{n}{2}\rfloor$. Thus $T(C_j)=C_j$ for $j \in[n-1]$.

If  $T(C_1)=C_{n-1}$, then  $T(C_{n-1})= C_{1}$ and by an argument  similar to the above we can conclude that $T(C_j)=C_{n-j}$ for all $j\in[n-1]$. This completes the proof.
\end{proof}

Fix $i\in [n-1]$. By Lemma \ref{CDk-01},  we have  $T(C(D_i))=C(D_j)$ for some $j\in [n-1]$. Since $C_i\subset C(D_i)$, by Lemma \ref{pc-4}, $T(C(D_i))=C(D_i)$ or $C(D_{n-i})$.  Using this observation and Lemma \ref{pc-4}, one can conclude the following.

\begin{lemma} \label{pc-4.1}
	Let $n \ge 3$ and  $T:\cS(n)\to \cS(n)$ be a semigroup isomorphism.  Then either
	\[ T(C(D_j)) = C(D_j) \text{ for all } 1 \le j < n, \text{ or } T(C(D_j))=  C(D_{n-j}) \text{ for all }1 \le j < n.\]
\end{lemma}

\begin{rem}\label{pcrem5.7}
Let $n \ge 3$ and  $T:\cS(n)\to \cS(n)$ be a semigroup isomorphism. By replacing $T$ with the map $X \mapsto P_n T(X) P_n$ if necessary, we may henceforth assume that
\begin{equation}\label{eqwlogn}
T(C_j)=C_j \quad \text{for all } 1 \le j < n.
\end{equation}
\end{rem}

We now focus on subgroups of $\cD(n)$. As a consequence of the above results, we show that $T$ induces an automorphism on the following subgroups in $\cS(n)$:
\begin{itemize}
	\item $\{aI_{k-1}\oplus D \oplus cI_{n-k-1}: a,c>0 \text{ and } D\in \cD(2)\}$ for $k \in [n-1]$.
	\item $\{aI_{k-1}\oplus bI_2 \oplus cI_{n-k-1}: a,b,c>0\}$ for $k \in [n-1]$.
\end{itemize}
\begin{lemma}\label{pclem3.12}
Let $n\geq 3$ be an integer and let $T:\cS(n) \to \cS(n)$ be an isomorphism satisfying \eqref{eqwlogn}. Define $\hat F_k:=\{aI_{k-1}\oplus D \oplus cI_{n-k-1}: a,c>0 \text{ and } D\in \cD(2)\}$ and $F_k=\{aI_{k-1}\oplus bI_2 \oplus cI_{n-k-1}: a,b,c>0\}$ for $k\in[n-1]$.  Then $T(\hat F_k)=\hat F_k$ and $T(F_k)=F_k$ for all $k\in[n-1]$.
\end{lemma}
\begin{proof}
 Fix $k\in [n-1]$. We first show that $T(\hat F_k)=\hat F_k$. Let $X=aI_{k-1}\oplus D \oplus cI_{n-k-1}\in \hat F_k$, where $D\in \cD(2)$ and $a,b>0$. Since $T(\cD(n))=\cD(n)$,  $T(X)=\diag(d_1,\ldots,d_n)$ with $d_j>0$ for all $j\in [n]$. For a non-diagonal matrix $Y \in \cS(2)$, define $\tilde Y_i:=I_{i-1}\oplus Y \oplus I_{n-i-1}$, where $i\in [n-1]$. By \eqref{eqwlogn},  $T(C_j)= C_j \text{ for all } 1 \le j < n$. Thus $\tilde Y_iT(X)=T(X)\tilde Y_i$ for all $i\in [n-1]\setminus\{k-1,k, k+1\}$ and for all non-diagonal $Y \in \cS(2)$.  This implies that $d_1=\cdots=d_{k-1}$ and  $d_{k+2}=\cdots=d_{n}$. Thus $T(X) \in \hat F_k$ and hence $T(\hat F_k)\subseteq \hat F_k$. Since $T$ is an automorphism, a similar argument shows that $T^{-1}(\hat F_k)\subseteq \hat F_k$. Thus $T(\hat F_k)= \hat F_k$ for $k\in [n-1]$.

To show $T(F_k)=F_k$, let $A=aI_{k-1}\oplus bI_2 \oplus cI_{n-k-1}\in F_k$. Note that $T(A)=\diag(d_1,\ldots,d_n) \in \cD(n)$ and  $\tilde Y_iT(A)=T(A)\tilde Y_i$ for all $i\in [n-1]\setminus\{k-1, k+1\}$ and for all non-diagonal $Y \in \cS(2)$, where the matrices $\tilde Y_i$ are defined as above.  This follows that  $d_1=\cdots=d_{k-1}$, $d_k=d_{k+1}$, and  $d_{k+2}=\cdots=d_{n}$. Thus $T(A) \in F_k$ for all $A\in  F_k$. Similarly, one can show that $T^{-1}(A) \in  F_k$ for all $A\in  F_k$. Hence  $T( F_k)= F_k$ for $k\in [n-1]$.
\end{proof}

The above lemmas yield the following theorem, which is crucial for the induction step.
\begin{theorem}\label{remgen}
	Let $n\geq3$ be an integer and let $T:\cS(n) \to \cS(n)$ be a semigroup isomorphism satisfying \eqref{eqwlogn}.  Then $T$ induces an automorphism on the following monoids:  
	\begin{enumerate}[(i)]
		\item   $\{\alpha I_{k-1}\oplus X \oplus \beta I_{n-k-1}: X\in \cS(2), \alpha, \beta>0\}$ for all $k\in[n-1]$.
		\item $\cS(n-1)\oplus S(1)=\{X \oplus \alpha : X \in \cS(n-1), \alpha> 0\}$.
	\end{enumerate}
\end{theorem}
\begin{proof}
Fix $k\in[n-1]$ and claim that $T(\{\alpha I_{k-1}\oplus X \oplus \beta I_{n-k-1}: X\in \cS(2), \alpha, \beta>0\})=\{\alpha I_{k-1}\oplus X \oplus \beta I_{n-k-1}: X\in \cS(2), \alpha, \beta>0\}$. By \eqref{eqwlogn}, $T(C_{k})=C_{k}$  and by Lemma \ref{pclem3.12}, $T(\hat F_{k})=\hat F_{k}$, where $\hat F_{k}=\{aI_{k-1}\oplus D\oplus bI_{n-k-1}: a,b>0 \text{ and } D\in \cD(2)\}$. Note that every $X\in \cS(2)$ can be factorize as $X=LDU$, where $L \in \cL_0(2),~D\in \cD(2)$ and $U\in \cU_0(2)$. Using the facts that $T$ is multiplicative and bijective, we conclude that $T(\{\alpha I_{k-1}\oplus X \oplus \beta I_{n-k-1}: X\in \cS(2), \alpha, \beta>0\})=\{\alpha I_{k-1}\oplus X \oplus \beta I_{n-k-1}: X\in \cS(2), \alpha, \beta>0\}$. 

 We next claim that  $T$  induces an automorphism on $\cS(n-1)\oplus S(1)$. Let $Z=X\oplus a$, where $X \in \cS(n-1), a> 0$. Let $\hat X=\frac{1}{a}X$. Then $Z=a(\hat X\oplus 1)$ and $T(Z)=\gamma(a)T(\hat X\oplus 1)$ by Remark \ref{rem1}. Since $\hat X\in \cS(n-1)$, by Whitney bidiagonal factorization, 

	\begin{equation}\label{bifact2}
	\hat X=\prod_{j=1}^{n-2} \prod_{k=n-2}^{j} \big{(} I_{n-1}+w_{j,k}E_{k+1,k} \big{)}~D \prod_{j=n-2}^{1} \prod_{k=n-2}^{j} \big{(} I_{n-1}+ w'_{j,k+1}E_{k,k+1} \big{)},
\end{equation}
where  $w_{j,k}\geq 0$ for $1\leq j\leq k \leq n-2,  w'_{j,k+1}\geq 0$ for $1\leq j\leq k\leq n-2$, and $D\in \cD(n-1)$. By Remark \ref{rem1}, $T(\cD{(n)})=\cD(n)$ and by the first part, $T$ maps $\{\alpha I_{k-1}\oplus X \oplus \beta I_{n-k-1}: X\in \cS(2), \alpha, \beta>0\}$ onto itself for all $k\in[n-1]$. Thus, using \eqref{bifact2} and the multiplicativity of $T$, we have $T(\hat X\oplus 1)\in \cS(n-1)\oplus \cS(1)$. Thus $T(Z) \in \cS(n-1)\oplus \cS(1)$, and hence  $T(\cS(n-1)\oplus \cS(1)) \subseteq \cS(n-1)\oplus \cS(1)$. Similarly, one can show that $T^{-1}(\cS(n-1)\oplus \cS(1)) \subseteq \cS(n-1)\oplus \cS(1)$. Thus $T(\cS(n-1)\oplus \cS(1)) = \cS(n-1)\oplus \cS(1)$.
\end{proof}

With these above ingredients in hand, we now complete the proof of Theorem \ref{multiplicative} for $n\geq 3$ using the induction hypothesis. Recall that the base case $n=2$ was proved in the preceding section. 
\begin{proof}[Proof of Theorem \ref{multiplicative} $(i)\implies (ii)$ for the General Case]
For the induction step, let $T:\cS(n) \to \cS(n)$ with $n\geq 3$ be an isomorphism. Replacing $T$ by the map $X \mapsto P_n T(X) P_n$ if necessary, we may assume that $T$ satisfies \eqref{eqwlogn}. By Theorem \ref{remgen}, $T$ induces an automorphism on the  monoid $\cS(n-1) \oplus \cS(1)$. Define $T_1: \cS(n-1) \rightarrow \cS(n-1)$ by
	\begin{equation}\label{pnceq001}
T_1(X) := \frac{1}{f(X)}Z,\quad \text{where} ~~T(X\oplus [1]):= Z \oplus [f(X)]. \end{equation}
	Note that $f:\cS(n-1)\to (0,\infty)$ is a multiplicative map. We now show that $T_1$ is a semigroup automorphsim on $\cS(n-1)$:
	\begin{itemize}
		\item \textbf{Multiplicative:} Let $X_1,X_2 \in  \cS(n-1)$ and let $T(X_1\oplus [1])= Z_1 \oplus [f({X_1})], T(X_2\oplus [1])= Z_2 \oplus [f({X_2})]$. Then $T_1(X_1)=\frac{1}{f({X_1})}Z_1$ and $T_1(X_2)=\frac{1}{f({X_2})}Z_2$. Since $T$ is multiplicative on $\cS (n)$, 
		\[T_1(X_1 X_2)= \frac{1}{f({X_1})f({X_2})}Z_1 Z_2= T_1(X_1) T_1 (X_2).\]
		
		\item \textbf{Injective:} Let $X_1,X_2 \in  \cS(n-1)$ such that $T_1(X_1)=T_1(X_2)$. If $T(X_1\oplus [1])= Z_1 \oplus [f({X_1})]$ and $ T(X_2\oplus [1])= Z_2 \oplus [f({X_2})]$, then $\frac{1}{f({X_1})}Z_1=\frac{1}{f({X_2})}Z_2$. Thus
		\begin{eqnarray*} 
			T(X_2\oplus [1])&=&\frac{f({X_2})}{f({X_1})}Z_1 \oplus [f({X_2})] \\\nonumber
			&=& \frac{f({X_2})}{f({X_1})}T(X_1 \oplus [1])\\ \nonumber
			&=&T\left(\gamma^{-1}\left(\frac{f({X_2})}{f({X_1})}\right)(X_1 \oplus [1])\right),
		\end{eqnarray*}
 by Remark \ref{rem1} (ii). Since $T$ is bijective,  $\gamma^{-1}\left(\frac{f({X_2})}{f({X_1})}\right)=1$, and so $X_1=X_2$.
		
		\item \textbf{Surjective:} Let $Y\in \cS(n-1)$. By Theorem \ref{remgen}, there exists $X_1\in \cS(n-1)$ and a positive real number $\alpha$ such that $T(X_1\oplus [\alpha])= Y \oplus [1]$. 
		 Let $X:=\frac{1}{\alpha}X_1$ and $Z:=\gamma\left(\frac{1}{\alpha}\right) Y$. By Remark \ref{rem1} $(ii)$,
		\[T (X\oplus[1]) = T(\frac{1}{\alpha} (X_1\oplus [\alpha]))
		= \gamma\left(\frac{1}{\alpha}\right)  (Y\oplus[1])
		= Z\oplus \left[\gamma\left(\frac{1}{\alpha}\right) \right],\]
and $f(X)=\gamma\left(\frac{1}{\alpha}\right) $. Thus there exists $X \in \cS(n-1)$ such that $T_1(X)=\frac{1}{f(X)}Z=Y$.
	\end{itemize}
This shows that $T_1:\cS(n-1)\to \cS(n-1)$ is an isomorphism.	By the induction hypothesis,
	\[T_1(X) = \gamma_{n-1} (\det X) (\det X)^{-\frac{1}{n-1}} RXR^{-1} \quad \text{for all } X\in \cS(n-1),\] where $\gamma_{n-1}:(0,\infty) \to (0,\infty)$ is a bijective multiplicative map and  $R = \sum\limits_{j=1}^{n-1} r_j E_{jj}$ or $\sum\limits_{j=1}^{n-1} r_jE_{j,n-j}\in \mathbb{R}^{{n-1}\times {n-1}}$ with $r_1,\ldots,r_{n-1}>0$. Hence, by \eqref{pnceq001}, we obtain
	\begin{equation}\label{indeq1}
	T (X\oplus[1])=f(X)\left(T_1(X)\oplus [1]\right)=f(X)\left( \gamma_{n-1} (\det X) (\det X)^{-\frac{1}{n-1}} RXR^{-1} \oplus [1]\right).\end{equation}
We now show that the case $R=\sum\limits_{j=1}^{n-1} r_jE_{j,n-j}$ is not possible.  Let $D_{\alpha}=\alpha I_{n-2}\oplus I_2$ with $\alpha \neq 1$. Suppose that $R=\sum\limits_{j=1}^{n-1} r_jE_{j,n-j}$. Then, by \eqref{indeq1}, \[ T (D_{\alpha})= f(\alpha I_{n-2}\oplus [1])\left( \gamma_{n-1}({\alpha}^{n-2})~ \alpha^{-\frac{n-2}{n-1}} \diag(1,\alpha,\ldots,\alpha)\oplus [1]\right),\] a contradiction, since Lemma~\ref{pclem3.12} implies that $ T(D_{\alpha})=aI_{n-2}\oplus bI_2$ for some $a,b>0$. Thus $R = \sum\limits_{j=1}^{n-1} r_j E_{jj}$. Let $\tilde R = R \oplus [1]$, and replace $T$ by the conjugated map $A \mapsto \tilde R^{-1}  T(A) \tilde R$. We may therefore assume that
 \begin{equation}\label{pcmaineq00}
 T(X \oplus [1]) = f(X)\left(\gamma_{n-1} (\det X) (\det X)^{-\frac{1}{n-1}} X \oplus [1]\right).
 \end{equation}
We next show that $T(A\oplus[1])= A\oplus [1]$ for all $A\in \cS_0(n-1)$.  Let $A\in \cS_0(n-1)$. Since $\gamma_{n-1}(1)=1$, by \eqref{pcmaineq00},  $T(A \oplus [1]) = f(A)\left(A \oplus [1]\right)$ for all $A\in \cS_0(n-1)$.  We claim that $f(A)=1$ for all $A \in \cS_0(n-1)$. By Whitney bidiagonal factorization it suffices to consider $A$ of the form 
	\begin{equation}\label{pceqn4.10}
	I_{k-1}\oplus X\oplus I_{n-k-1},
	\end{equation}
	with $X$ in the one-parameter Chevalley subgroups
	\[
	\{ \exp(a E_{12}) = \begin{pmatrix} 1 & a \\ 0 & 1 \end{pmatrix} \, : \, a > 0 \} \qquad \text{or} \qquad
	\{ \exp(a E_{21}) = \begin{pmatrix} 1 & 0 \\ a & 1 \end{pmatrix} \, : \, a > 0  \}
	\]
	or $X$ a torus element $\begin{pmatrix} a & 0 \\ 0 & 1/a \end{pmatrix}$ for $a>0$ -- and with $k \in [n-2]$.
 We next show that $f(A)=1$ for all $X$ as above.  For all these $X$, $f$ is a  function of  $a$ only. Thus we rewrite it as $f_u(a)$ for $X=\begin{pmatrix}
	1 & a\\ 0 & 1
	\end{pmatrix}$,  $f_l(a)$ for $X=\begin{pmatrix}
	1 & 0\\ a & 1
	\end{pmatrix}$, and $f_d(a)$ for $X=\begin{pmatrix}
	a & 0\\ 0 & \frac{1}{a}
	\end{pmatrix}$. It is easy to verify that $f_d:(0, \infty)\to (0,\infty)$ is mutiplicative i.e. $f_d(a_1 a_2)=f_d(a_1)f_d(a_2)$ for all $a_1,a_2>0$. We next show that  $f_u(a_1 + a_2)=f_u(a_1)f_u(a_2)$ for all $a_1,a_2>0$. Let $a_1, a_2>0$ and take $X_1= \begin{pmatrix}
	1 & a_1\\ 0 & 1
	\end{pmatrix}, X_2=\begin{pmatrix}
	1 & a_2\\ 0 & 1
	\end{pmatrix}$, and $X_3=\begin{pmatrix}
	1 & a_1+a_2\\ 0 & 1
	\end{pmatrix}$. Then $X_1 X_2=X_3$ and \[ T(I_{k-1}\oplus X_3\oplus I_{n-k-1})  =T\left( (I_{k-1}\oplus X_1\oplus I_{n-k-1}) (I_{k-1}\oplus X_2\oplus I_{n-k-1})\right),\]
where $k\in [n-2]$. Since $T(A \oplus [1]) = f(A)\left(A \oplus [1]\right)$ for all $A\in \cS_0(n-1)$, we have  \begin{equation}\label{pceq4.11}
f_u(a_1 + a_2)=f_u(a_1)f_u(a_2).
\end{equation}
 Similarly, we can show that $f_l(a_1 + a_2)=f_l(a_1)f_l(a_2)$ for all $a_1, a_2>0$. We now show that $f_u(a)=1$ for all $a>0$.  Let $a>0$. Note that for all $p>0$,
	
	\[\begin{pmatrix}
		p & 0\\ 0 & \frac{1}{p}
	\end{pmatrix} \begin{pmatrix}
	1 & a\\ 0 & 1
	\end{pmatrix} \begin{pmatrix}
	\frac{1}{p} & 0\\ 0 & p
	\end{pmatrix}=\begin{pmatrix}
	1 & p^2a\\ 0 & 1
	\end{pmatrix}.\]
By embedding all the above matrices into $n \times n$ matrices of the form \eqref{pceqn4.10} and applying $T$, we have \begin{equation}\label{pceq4.12}
f_d(p)f_u(a)f_d\left(\frac{1}{p}\right)=f_u(p^2a).
\end{equation} Since $f_d$ is multiplicative, $f_u(a)=f_u(p^2a)$ for all $p>0$. In particular for $p=2$,  $f_u(a)=f_u(4a)$.  Using the identity \eqref{pceq4.11}, we have
\begin{eqnarray*}
	f_u(a)&=&f_u(4a)\\ &=& f_u(2a)f_u(2a)\\
	 &=& (f_u(a))^4.
\end{eqnarray*}
Thus $f_u(a)=1$ for all $a>0$. Similarly, one can show that $f_l(a)=1$ for all $a>0$. Finally we show that $f_d(p)=1$ for all $p>0$. Let $a > 0$ and $X := \begin{pmatrix}
1 & a\\ 0 & 1
\end{pmatrix}$. Then 
$XX^T = Y^TZY$, where $Y  = \begin{pmatrix}
	1 & \frac{a}{a^2+1}\\ 0 & 1
\end{pmatrix}$ and $Z =\begin{pmatrix}
a^2+1 & 0\\ 0 & \frac{1}{a^2+1}
\end{pmatrix}$. Define 
\[\hat X:=\begin{pmatrix}
	I_{k-1} & 0& 0\\ 0 & X & 0\\0 & 0& I_{n-k-1}
\end{pmatrix}, ~\hat Y:=\begin{pmatrix}
I_{k-1} & 0& 0\\ 0 & Y & 0\\0 & 0& I_{n-k-1}
\end{pmatrix},~ \hat Z:=\begin{pmatrix}
I_{k-1} & 0& 0\\ 0 & Z & 0\\0 & 0& I_{n-k-1}
\end{pmatrix},\]
where $k\in [n-2]$. Since $f_u(p)=f_l(p)=1$ for all $p>0$, we have $T(\hat X) = \hat X$, $T(\hat X^T) = \hat X^T$, $T(\hat Y) = \hat Y$, and $T(\hat Y^T) = \hat Y^T$. This implies that
\[\hat X {\hat X}^T = T(\hat X {\hat X}^T) = T({\hat Y}^T\hat Z\hat Y) = T({\hat Y}^T)T(\hat Z) T(\hat Y) = {\hat Y}^T f_d(a^2+1)\hat Z \hat Y = f_d(a^2+1) \hat X{\hat X}^T.\]
Thus $f_d(a^2+1) = 1$ for all $a> 0$. Since $f_d$ is multiplicative, $f_d(a) = 1$ for all $a>0$. Hence $f(A)=1$ for all $A \in S_0(n-1)$ and
\begin{equation}\label{pcmaineq1}
T(A\oplus [1] )=A \oplus [1]\quad  \text{for all}\quad A \in S_0(n-1). 
\end{equation}
	
Next, we turn our attention to matrices of the form $I_{n-2}\oplus A \in \cS(n)$.  Notice that $T$ induces a semigroup automorphism on
	$\{aI_{n-2} \oplus X: a>0, X \in \cS(2)\}$ by Theorem \ref{remgen}.
	Define $ T_2: \cS(2) \rightarrow \cS(2)$ by 
	\[T_2(X): = \frac{1}{g(X)} Z,\quad  \text{where } T (I_{n-2} \oplus X) := g(X)I_{n-2} \oplus Z,\] 
	where $g:\cS(2)\to (0,\infty)$ is a multiplicative map. We now show that $ T_2$ is a semigroup automorphsim on $\cS(2)$. Using a similar argument as above, it follows that $ T_2$ is multiplicative. It remains to show that $ T_2$ is bijective.
	\begin{itemize}
		\item \textbf{Injective:} Let $X_1,X_2 \in  \cS(2)$ such that $ T_2(X_1)= T_2(X_2)$. If $T(I_{n-2} \oplus X_1)= g(X_1)I_{n-2} \oplus Z_1$ and $ T(I_{n-2} \oplus X_2)= g(X_2)I_{n-2} \oplus Z_2$, then $\frac{1}{g(X_1)}Z_1=\frac{1}{g(X_2)}Z_2$. Thus
		\begin{eqnarray}
			 T(I_{n-2} \oplus X_2)&=& g(X_2) I_{n-2}\oplus \left(\frac{g(X_2)}{g(X_1)}\right)Z_1 \nonumber\\
			&=& \frac{g(X_2)}{g(X_1)} T(I_{n-2} \oplus X_1) \nonumber\\
			&=& T\left(\gamma^{-1}\left(\frac{g(X_2)}{g(X_1)}\right)(I_{n-2} \oplus X_1)\right), \nonumber
		\end{eqnarray}
	by Remark \ref{rem1} (ii). This implies $\gamma^{-1}\left(\frac{g(X_2)}{g(X_1)}\right)=1$ and $X_1=X_2$.
		
		\item \textbf{Surjective:} Let $Y\in \cS(2)$. Since $T$ is bijective on
		$\{aI_{n-2} \oplus X: a>0, X \in \cS(2)\}$,  there exists $X_1\in \cS(2)$ and a positive real number $\alpha$ such that $T(\alpha I_{n-2} \oplus X_1)= I_{n-2} \oplus Y$. Since $ T(\frac{1}{\alpha}I_n)= \beta I_n$ for some $\beta>0$, define $X:=\frac{1}{\alpha}X_1$ and $Z:=\beta Y$. Then 
		\begin{eqnarray*}
			 T(I_{n-2} \oplus X) =T\left(\frac{1}{\alpha} (\alpha I_{n-2}\oplus X_1)\right)= \beta I_{n-2} \oplus Z.
		\end{eqnarray*}
		Thus there exists $X \in \cS(2)$ such that $T_2(X)=\frac{1}{\beta}Z=Y$.
	\end{itemize}
	This implies that $T_2$ is an automorphism on $\cS(2)$. By the induction hypothesis, $T_2$ has the form
	$X \mapsto  \gamma_{2} (\det X) (\det X)^{-\frac{1}{2}}  SXS^{-1}$, where $\gamma_2:(0,\infty) \to (0,\infty)$ is a bijective multiplicative map and $S =\begin{pmatrix}
		s_1 & 0\\ 0 & s_2
	\end{pmatrix}$ or $\begin{pmatrix}
0 & 	s_1\\ 	s_2 & 0
	\end{pmatrix}$ with $s_1,s_2>0$. Thus for all $X\in \cS(2)$,
	 \begin{equation}\label{maineq3}
	  T(I_{n-2} \oplus X)=g(X)\left(I_{n-2} \oplus \gamma_{2} (\det X)(\det X)^{-\frac{1}{2}} SXS^{-1}\right).
	  \end{equation}
We next show that $S=\begin{pmatrix}
	0 & 	s_1\\ 	s_2 & 0
\end{pmatrix}$ is infeasible. Let \[A=I_{n-3}\oplus \begin{pmatrix}
		1 & 0 & 0\\1 & 1 & 0\\0 & 0& 1
	\end{pmatrix}\quad \text{and}\quad B= I_{n-3}\oplus\begin{pmatrix}
	1 & 0 & 0\\0 & 1 & \alpha \\0 & 0& 1
	\end{pmatrix},\] with $\alpha\neq 0$. Then $AB=BA$ and by \eqref{pcmaineq1}, $ T(A)=A$. If $S = \begin{pmatrix}
	 	0 & s_1\\  s_2 & 0
	 \end{pmatrix}$ then using \eqref{maineq3} we have $T(B)=g\left(\begin{pmatrix}
	 1 & \alpha \\ 0 &1
	 \end{pmatrix}\right)\left(I_{n-3}\oplus\begin{pmatrix}
	 1 & 0 & 0\\0 & 1 & 0 \\0 & \frac{s_2}{s_1}\alpha & 1
	 \end{pmatrix}\right)$, and so $T(A)  T(B)\neq  T(B)  T(A)$, a contradiction. Thus $S =\begin{pmatrix}
	 	 s_1 & 0\\ 0 & s_2
	 \end{pmatrix}$ for some $s_1, s_2>0$. Let $S^\prime:=\diag (1,\frac{s_2}{s_1})$. Then $SXS^{-1}=S^\prime X {S^\prime}^{-1}$ for all $X \in \mathbb{R}^{2\times 2}$.  Define $\tilde S := \begin{pmatrix}
	 I_{n-2} & 0\\ 0 & S^\prime
	 \end{pmatrix}\in \cD(n)$. By replacing $ T$ with the map $A \mapsto \tilde S^{-1}  T(A) \tilde S$, we may  assume that \begin{equation}\label{pceq4.13}
	 T(I_{n-2} \oplus X)= g(X)\left( I_{n-2}\oplus \gamma_2 (\det X) (\det X)^{-\frac{1}{2}} X\right)\quad \text{for all } X \in \cS(2).
	  \end{equation}
We now show that $T(I_{n-2} \oplus X)= I_{n-2} \oplus X$ for all $X \in \cS_0(2)$. By a similar argument as above, one can show that $g(X)=1$ for $X=\begin{pmatrix}
		1 & \alpha \\ 0 & 1
	\end{pmatrix}, \begin{pmatrix}
		1 & 0 \\ \alpha & 1
	\end{pmatrix}$ and $\begin{pmatrix}
		\alpha & 0 \\ 0 & \frac{1}{\alpha}
	\end{pmatrix}$ for all $\alpha>0$. Again by Whitney's bidiagonal factorization, $g(X)=1$ for all $X\in \cS_0(2)$. Since $\gamma_2(1)=1$, by \eqref{pceq4.13}, we have
	
	\begin{equation}\label{pcmaineq2}
		T(I_{n-2} \oplus X)=  I_{n-2}\oplus  X\quad \text{for all}\quad X \in \cS_0(2).
		 \end{equation}
Thus far, we have shown that $T$ fixes every upper and lower elementary bidiagonal matrix. We now prove that $T(D)=D$ for all $D\in \cD_0(n)$. Let $D = \diag(a_1,\ldots,a_n)\in \cD_0(n)$. Then $\prod\limits_{i=1}^{n} a
_i= 1$ and $D = D_1 D_2$ with $D_1 = \diag\left(a_1,\ldots,a_{n-2},\frac{1}{\prod\limits_{i=1}^{n-2}a_i},1\right)$ and $D_2 = I_{n-2}\oplus \diag(\frac{1}{a_n}, a_n)$. By \eqref{pcmaineq1} and \eqref{pcmaineq2}, we have $ T(D)=T(D_1)T(D_2)=D_1D_2=D$. Thus 
	
		\begin{equation}\label{pcmaineq3}
		T(D)=  D  \quad \text{for all }\quad D \in \cD_0(n).
	\end{equation}
Now note that all the bidiagonal factors of an $n\times n$ invertible TN matrix, with determinant $1$, are either of the form $I_{k-1}\oplus X \oplus I_{n-k-1}$ with $\det(X)=1$, or a diagonal matrix with determinant $1$. Thus  $T(A)=A$ for all $A \in \cS_0(A)$. 

To  summarize the above analysis, we have used the maps $X\mapsto P_nT(A)P_n$ (see Remark \ref{pcrem5.7}), $X \mapsto \tilde R^{-1}  T(X) \tilde R$ and $X \mapsto \tilde S^{-1}  T(X) \tilde S$ to conclude that $T(A)=A$ for all $A \in \cS_0(A)$, where

\[P_n:=\begin{pmatrix}
	0 & \cdots & 1 \\
	\vdots & \iddots & \vdots \\
	1 & \cdots & 0 \\
\end{pmatrix}, \quad \tilde R=\begin{pmatrix}
R & 0\\ 0 & 1
\end{pmatrix}, \quad \tilde S=\begin{pmatrix}
I_{n-1} & 0\\ 0 & s
\end{pmatrix} \text{ with } R\in \cD(n-1) \text{ and } s>0. \]
	
To conclude the proof, write a general $A \in \cS(n)$ as $A = ((\det A)^{\frac{1}{n}} I_n)A_0$ with $\det A_0 = 1$. Thus
	\begin{eqnarray*}
		T(A) &=& T((\det A)^{\frac{1}{n}} A_0)\\
		&=& \gamma((\det A)^{\frac{1}{n}})  T(A_0)\\
		&=& \gamma((\det A)^{\frac{1}{n}})  A_0\\
		&=&\mu(\det A)(\det A)^{-\frac{1}{n}} A ,
	\end{eqnarray*}
where $\mu(x):=\sqrt[n]{\gamma(x)}$ is a bijective multiplicative map on $(0,\infty)$.  Hence 
\[T(A)=\mu(\det A)(\det A)^{-\frac{1}{n}} A \text{ for all } A\in \cS(n).\]
 This completes the proof.
\end{proof}

Using Theorem \ref{multiplicative}, we now prove our second main result.
 
 \begin{proof}[Proof of Theorem \ref{multiplicative2}]
That $(ii)\implies (i)$ is trivial by hypothesis. For $(i)\implies (ii)$,  using Proposition \ref{prop2},  $T$ can be extended to a monoid automorphism $\hat T$ of $ITN(n)$ such that $\hat T \equiv T$ on $G(n)$. By Theorem \ref{multiplicative}, we have the desired structure of $\hat T$ and hence of $T$.
 \end{proof}

We conclude with the following remark.
\begin{rem}
  Let $G$ and $H$ be two groups. If $\phi$ is an automorphism of $G$ and $\psi$ is an automorphism of $H$, then the product function $\phi \times \psi: G \times H \to G \times H$ defined by
\[(\phi \times \psi )(g, h) := (\phi (g), \psi(h))\quad \text{for all } g\in G \text{ and } h\in H,\]
is an automorphism of $G \times H$. But it is not true in general that every automorphism of $G \times H$ has the above form. 

Now we see that  every automorphism of the direct product of the group $Z^+(n)=\{aI_n: a>0\}$ and submonoid $\cS_0(n)$ has the above factorization.  Note that $\cS(n)= Z^+(n) \times  \cS_0(n)$. By Theorem \ref{multiplicative}, if $T$ is an automorphism of $\cS(n)$ then $T_1=T |_{Z^+(n)}$ and $T_2=T|_{\cS_0(n)}$ are automorphisms on $Z^+(n)$ and $\cS_0(n)$, respectively. Thus $T=T_1\times T_2$. 
\end{rem}

\section*{Acknowledgements}
	We thank Apoorva Khare for a detailed reading of an earlier draft and for providing valuable feedback. This project was started while the authors were attending the workshop “Theory and applications of total positivity", held at the
American Institute of Mathematics (AIM), from July 24-28, 2023. The authors thank AIM and the NSF for their financial support of AIM and the participants in this workshop.  P.N. Choudhury was partially supported by supported by ANRF Prime Minister Early Career Research Grant ANRF/ECRG/2024/002674/PMS (ANRF, Govt.~of India), INSPIRE Faculty Fellowship research grant DST/INSPIRE/04/2021/002620
(DST, Govt.~of India), and IIT Gandhinagar Internal Project: IP/IP/50025. S.M. Fallat was supported in part by an NSERC Discovery Research Grant 2025-05272, Application No.: RGPIN-2019-03934.

\end{document}